\documentclass{amsart}
\usepackage{graphicx}
\usepackage{amsthm,amsmath,verbatim,amssymb}
\usepackage{color}

\newcommand{\kahler}{K\"ahler }

\newcommand{\R}{{\mathbb R}}
\newcommand{\C}{{\mathbb C}}

\renewcommand{\d}{\partial}

\newcommand{\dbar}{\bar\partial}
\newcommand{\ddbar}{\partial\dbar}

\newcommand{\D}{{\mathbf D}}

\newcommand{\half}{{\frac{1}{2}}}

\renewcommand{\phi}{\varphi}

\newcommand{\E}{{\mathbf E}\,}
\newcommand{\F}{\mathbf F}

\newtheorem{thm}{Theorem}

\renewcommand{\d}{\partial}
\newcommand{\Hn}{H^0(M,L^n)}

 \def   \half   {{\frac{1}{2}}}
\def \to {\rightarrow}

\begin{document}\title[Critical values]{Critical values of fixed Morse index of random analytic functions on Riemann surfaces}

\author[Renjie Feng]{Renjie Feng}
\author[Steve Zelditch]{Steve Zelditch}
\address{Department of Mathematics and Statistics, University of Maryland College Park,USA}
\email{renjie@math.umd.edu}
\address{Department of Mathematics, Northwestern University, USA}
\email{zelditch@math.northwestern.edu}

\thanks{Research partially supported by NSF grant  DMS-1206527.}
\maketitle
\date{\today}
 
This note is an addendum to \cite{FZ}. In that 
  article, we determined the limit distribution of  {\it critical values} of pointwise norms $|s_n(z)|_{h^n}$ of 
$L^2$-normalized random global holomorphic sections  of $\Hn$ where $(M, \omega)$ is a \kahler manifold of complex  dimension $m$  and $(L, h)
\to (M, \omega)$ is a positive holomorphic line bundle of degree 1 whose
curvature form is $\omega = \rm{Ric} (h)$. The line bundle $L^n$ is the $n$th power of $L$, so
that  $\Hn$  is  analogous to the space of polynomials of degree $n$ and
the Hermitian metric $h^n$ on $L^n$ is the nth tensor power of $h$. We studied two probability measures on $\Hn$:
(i) a certain canonical ``normalized Gaussian measure" induced by $h$, for which 
$\E ||s||_{L^2} = 1$; and (ii) the Euclidean surface  `spherical measure' on the unit sphere
$S \Hn$ where $||s||_{L^2} = 1$. The motivation for use of the spherical
measure was to count a critical value just once in a family of dilates $\{cs: c \in \C\}$.
  In this note, we add two results on the large $n$ asymptotics of the density of critical values
of norms of sections in  $\Hn$  to \cite{FZ} which clarify the nature of the results. 

The first addition is to make more precise  Remark 2 of \cite{FZ}, where it was   explained how to compute the distribution of the critical values at critical points of $|s_n(z)|_{h^n}$  of
fixed Morse index. The calculation  in general complex  dimension $m$ is a  complicated Kac-Rice integral over complex
symmetric matrices of fixed index and rank $m$ in complex dimension $m$. But in complex dimension one
(i.e. on Riemann surfaces),
it  is simple enough to evaluate
explicitly. The first 
 purpose of this note is to supply the calculation   of the limit distribution of
critical values at local maxima or at saddles on Riemann surfaces, completing Remark 2 of \cite{FZ}.

The second addition   is the calculation of the  second order term in the large $n$ asymptotic expansion of the expected density of critical values on Riemann surfaces. The result is that the second term is a topological term. It follows
that the large $n$ expansion is universal to two orders. In \cite{DSZ} the same kind of universality was
shown for the density of critical points on $M$.

Throughout we assume familiarity with the notation and terminology
of \cite{FZ}. To keep the article to an appropriate length, we only review
notation and results that are needed to state and prove the new results.

% we will only carry out the calculation on the Riemann surfaces for simplicity, the same 
%calculations apply to higher dimentional. 

\subsection{Density of critical values at local maxima and saddles}
In \cite{DSZ, DSZ1} the authors determined the distribution of critical points of
a fixed Morse index, and the present discussion takes off from that point. 
Note that  $d|s_n|^2_{h_n}=0\Leftrightarrow \nabla_{h^n}s_n=0 \,\, \mbox{or}\,\, s_n=0$.
Local minima of $|s_h|^2$ are necessarily zeros, and thus are trivial from
the viewpoint of critical values.
As in  \cite{DSZ1},  the topological index of a section $s$ at a critical point $z_0$  is defined to be  the index of the
vector field $\nabla_h s$ at $z_0$ (where $\nabla_h s$ vanishes).  
Critical points of a
section $s$ in dimension one are (almost surely) of topological
index $\pm 1$.
The critical points of $s$ of index $1$ are the saddle points of
$\log |s|_h$ (or equivalently, of $|s|^2_h$), while those of
topological index $-1$ are local maxima of $\log|s|_h$ in the case
where $L$ is positive. In complex dimension one, and with $(L, h)$ a positive line bundle,    topological index
$1$ corresponds to $\log|s|_h$ having Morse index 1, while
topological index $-1$ corresponds to Morse index 2.
In \cite{DSZ1} it is shown that in complex dimension one, the number of 
saddle points of $|s|_{h^n}^2$ is asymptotically $\frac{4}{3} n$, while
the number of local maxima is asympotically $\frac{1}{3} n$.

Let $\omega = \frac{i}{2} \ddbar \log h$ be the K\"ahler  metric  associated to $(L,h)$. Thus,  $\frac{1}{\pi}\omega$ is a
de Rham representative of the Chern class $c_1(L)\in H^2(M,\R) $\footnote{In \cite{FZ}, the area form was defined to be $  \omega/\pi $. In this article we define it to be   $ \omega$. As a result, there  are slight differences in normalization between this article and \cite{FZ}.}.

In any dimension $m$, we define the empirical measure of nonvanishing critical values of index $k$ as
$$CV^{m,k}_{s_n}=\frac 1{n^m}\sum_{z:\,\, \nabla_{h^n}s=0 \,\,\mbox{with index  } k }
\delta_{|s_n|_{h^n}}$$
and define the expected distribution of such critical values by
$$\D_n^{m,k}(x):=\E (CV^{m,k}_{s_n}).$$
As seen below, it  is a measure with a smooth density on $\R$.
Then as proved in \cite{FZ} (see Remark 2), 

\begin{thm}\cite{FZ} For both the normalized Gaussian ensemble and spherical ensemble,  the universal limit as $n \to \infty$ for the expected density of nonvanishing critical values of index $k$, 

 $$\lim_{n\to \infty }\D^{m,k}_n(x)=  f_{m,k}(x)xe^{-\pi^mx^2},$$
 where $$f_{m,k}(x)=c_m\int_{S_{k,x}(\C^m)} e^{-\pi^m|\xi|^2}\left||\sqrt P \xi|^2-x^2 I\right|d\xi$$
 where $$S_{k,x}(\C^m)=\{\xi \in S(\C^m):\,\,\rm{index}(|\sqrt P \xi|^2-x^2I)=k \}.$$

\end{thm}

We now  evaluate the integrals
in complex dimension 1.

In the case of Riemann surfaces, the matrix $P=2$. Assuming the degree of
$L$ equals 1, the area $\int_M \omega = \pi.$
Then the limit of the expected density of local maxima of $|s|_{h^n}$ for $s \in H^0(M, L^n)$ 
 is  given by the universal formula,
\begin{align*}\lim_{n\to\infty}\D_n^{1,-1}(x)&=x\left(\frac{2}{\pi}\int_{\xi\in \mathbb C:\,|\xi|^2<\frac{x^2}{2}} e^{-\pi |\xi|^2}(x^2-2|\xi|^2)d\xi\right)e^{-\pi x^2}\\&= x   \left(\frac 2\pi x^2-\frac4{\pi^2}+\frac4{\pi^2}e^{-\frac{\pi x^2}2}\right) e^{-\pi x^2}. \end{align*}

\begin{center} \label{graph of  $f_1(x)e^{-x^2}$}
\includegraphics[scale=0.15]{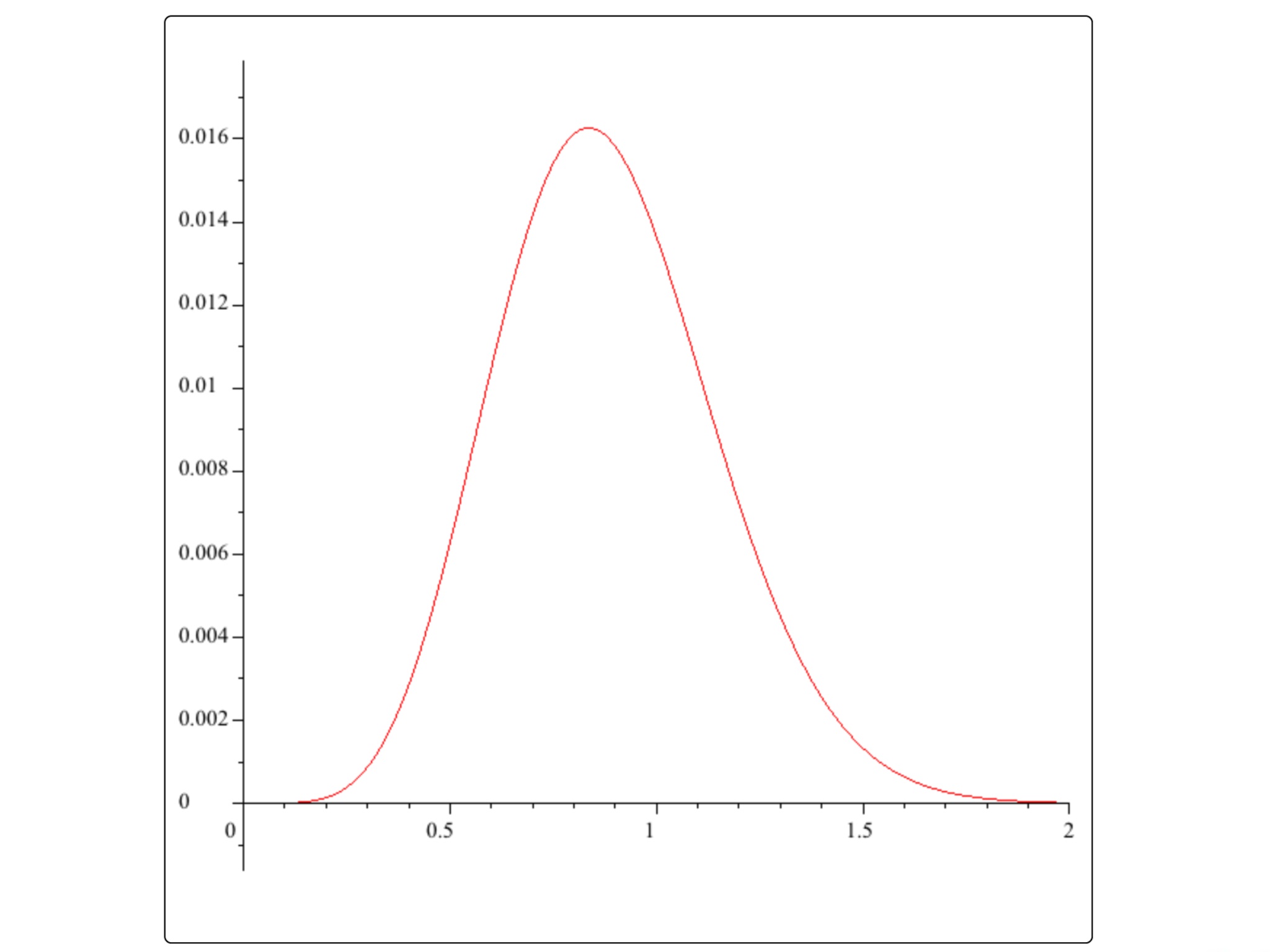}\par\text{Graph of  $\D^{1,-1}_{\infty}(x)$ in dimension one}
\end{center}
\bigskip

Moreover, the  expected density of saddle values has the universal limit, 
\begin{align*}\lim_{n\to\infty}\D_n^{1,1}(x)&=x\left(\frac{2}{\pi}\int_{\xi\in\mathbb C:\,|\xi|^2>\frac{x^2}{2}} e^{-\pi |\xi|^2}(2|\xi|^2-x^2)d\xi\right)e^{-\pi x^2}\\&=\frac{4x}{\pi^2} e^{-\frac{3\pi}2 x^2}. \end{align*}

\begin{center} \label{graph of  $f_1(x)e^{-x^2}$}
\includegraphics[scale=0.15]{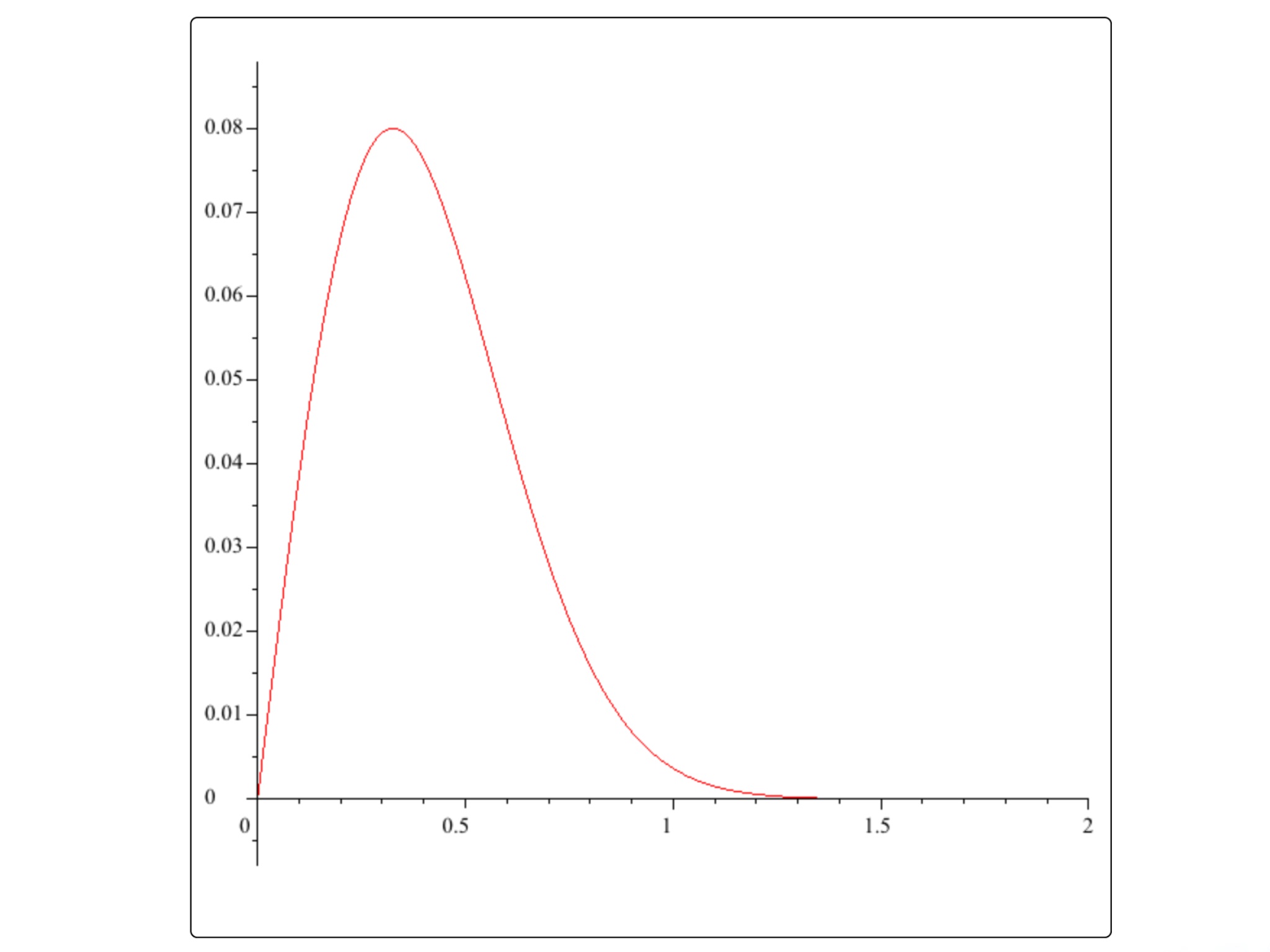}\par\text{Graph of  $\D^{1,1}_{\infty}(x)$ in dimension one}
\end{center}
\bigskip

In both cases, the density of critical values $x$ at saddles and local maxima
vanishes to order one at $x = 0$.   Although we take the norm (or modulus) when defining
critical values, the density reflects the fact that  the sections are complex valued
and $x$ plays the role of  the polar coordinate.  

\subsection{The second order topologically invariant term }

Theorems 1.1 and 1.5 of \cite{FZ} give the leading term in the density of critical values. In fact,
the proof shows that there exists   a full expansion $$\D_n=\D_\infty+\frac1n \F_\infty+\cdots $$
in powers of $n$. 
The leading order term $\D_{\infty}$ is calculated in \cite{FZ} and shown to be universal.
 The  second addition of this note  is to prove the topological
invariance of the second term in complex dimension one:

\begin{thm} For both the  normalized Gaussian and spherical ensembles, the second order term is a topological invariant of \kahler metrics on Riemann surfaces, 
$$\F_\infty(x)=-\frac { \chi(M) \pi^2x}{4 } [\int_{\C} e^{-\frac\pi 2|\xi|^2-\pi x^2}(\pi |\xi|^2-2) \left||\xi|^2- x^2 \right| d\xi] $$
where $\chi(M)$ is the Euler characteristic of the Riemann surfaces.
\end{thm}

\begin{proof}  The Kac-Rice formula for $\D_n$ is given in Section 4 (page 665 of \cite{FZ}) in terms of 
a certain function $p_z^n(x, \theta, 0, \xi)$ given in (5.2) and Lemma 5.1. To determine the full
expansion it is only necessary to find the expansion of the matrices $A_n$
and  $\Lambda_n$ (see Section 7.2). The entries of the matrices are various
derivatives of the Bergman kernel $B_n(z,w )$ on the diagonal. We recall
the TYZ expansion (cf. \cite{Lu}),
\begin{equation} \label{bergman}B_n(z, z) = \frac{1}{\pi^m}  n^m e^{n \phi(z)} \left[1+a_1(z) n^{-1}
+a_2(z) n^{-2}+\cdots\right]\,,\end{equation}
where $a_1 = \half S$ is half of the complex scalar curvature $S$ of $\omega$,
 and is a quarter of the scalar curvature of the Riemannian metric\footnote{ We thank Z. Lu for clarifying the coefficient of $a_1$ in
\eqref{bergman}.
}. Also, $h = e^{-\phi}$ in a local frame.
Denote by  $d_n$  the dimension of $\Hn$.  By  Riemann-Roch, in complex
dimension one, 
$d_n= n  + 1 - g = n + \half \chi(M), $
where $\chi(M) = 2 - 2 g$ is the Euler characteristic. This also follows by
integrating the $e^{- n \phi}$ times \eqref{bergman}  (i.e. the Szeg\"o kernel
$\Pi_n(z, z)$) and applying Gauss-
Bonnet.

In Section 7 of \cite{FZ}, we used the estimates of covariance matrix to get the expression of the leading term 
$\D_\infty$. Here, we  continue with the estimates to get the expression for $\F_\infty$ on Riemann surfaces.
We only keep track of terms that
contribute to the second order term and ignore all negligible terms.
Recall in Sections 4 and 5 of \cite{FZ}, we have the following formula for $\D_n$ in the case of Riemann surfaces with $m=1$,

$$\D_n(x)=\frac {2x}{\pi^2n } \int_{M}\int_{\C} \frac {e^{ -\left\langle \begin{pmatrix} \xi  \\y\end{pmatrix}, \Lambda_n^{-1}\begin{pmatrix} \bar \xi \\ \bar y \end{pmatrix}\right \rangle }}{  A_n \det \Lambda_n}\left|  |\xi|^2-n^2  x^2 \right| d\xi  dV_\omega,$$
%where $A_n$ and $\Lambda_n$ are matrices given in Section 7 of \cite{FZ} and $x=|y|$.

where 

$$A_n=\frac{n^{}}{\pi d_n}\left(nI+a_1 I+n^{-1}(a_2 I+\d_j\bar \d_{j'}a_1+\cdots)\right)$$
and by \eqref{bergman}, $$\begin{aligned} \Lambda_n& =\frac{n^{}}{\pi d_n} \begin{pmatrix}
 (2n^2-nS)(1+n^{-1}a_1+\cdots)    &n^{-1}\d_j \d_q a_1+O(n^{-2})
\\ n^{-1}\bar\d_j\bar \d_q a_1+O(n^{-2})
&1+n^{-1}a_1+O(n^{-2})
\end{pmatrix} 
 \end{aligned}.$$
 
We change variable $\xi\to n\xi$ to rewrite, 

\begin{equation}\label{dt}\D_n(x)=\frac {2n^3x}{\pi^2 } \int_{M}\int_{\C} \frac {e^{ -\left\langle \begin{pmatrix} \xi  \\y\end{pmatrix}, \frac{\pi d_n}{n}\tilde\Lambda_n^{-1}\begin{pmatrix} \bar \xi \\ \bar y \end{pmatrix}\right \rangle }}{  A_n \det \Lambda_n}\left|  |\xi|^2- x^2 \right| d\xi  dV_\omega.\end{equation}
Here,  $\tilde\Lambda_n$ has the full expansion  
$$\begin{aligned} \tilde\Lambda_n=\Lambda^0+n^{-1}\Lambda^1+\cdots 
  \end{aligned}$$
 with
  $$ \Lambda^0= \begin{pmatrix} 2&0\\0&1  \end{pmatrix},\;\;\;\Lambda^1= \begin{pmatrix} 2a_1-S &0\\0&a_1  \end{pmatrix}= \begin{pmatrix} 0 &0\\0&a_1  \end{pmatrix}$$
where we use the fact that the second term $a_1$  in TYZ expansion equals to $\frac12 S$.

 Thus $$\begin{aligned} \tilde\Lambda_n^{-1}=\begin{pmatrix} \frac12&0\\0&1  \end{pmatrix} -n^{-1}\begin{pmatrix} 0 &0\\0&a_1  \end{pmatrix}+\cdots 
  \end{aligned}$$

It follows that,
$$e^{ -\left\langle \begin{pmatrix} \xi  \\y\end{pmatrix}, \frac{\pi d_n}{n}\Lambda_n^{-1}\begin{pmatrix} \bar \xi \\ \bar y \end{pmatrix}\right \rangle }$$$$ =e^{-\frac\pi 2|\xi|^2-\pi x^2}\left(1-\pi n^{-1}[ \frac14\chi(M) |\xi|^2+(\frac12\chi(M)-a_1)x^2]+\cdots \right).$$

If we substitue the asymptotic expansions  of $A_n$ and $\Lambda_n$ into the equation \eqref{dt},  and only keep track of  terms of order $n^{-1}$,  we get

$$\F_\infty=-\pi^2  x \int_{M}\int_{\C} e^{-\frac\pi 2|\xi|^2-\pi x^2}\left( [\frac14 \chi(M) |\xi|^2+( \frac12\chi(M)-a_1)x^2]\right) $$$$||\xi|^2- x^2  | d\xi  dV_\omega
+\pi x \int_{M}\int_{\C} (\frac 32 \chi(M)-3a_1+\frac12 S)e^{-\frac\pi 2|\xi|^2-\pi x^2}||\xi|^2- x^2|  d\xi  dV_\omega.$$Recall Gauss-Bonnet Theorem 
$\frac\pi2\chi(M)=\int_M a_1 \omega$ again,  combine this with the assumption that the volume of $M$ is $\pi$, we have $\int_{M}(\frac12\chi(M)-a_1) \omega=0$. Thus we can simplify the above expression as, 

$$\F_\infty(x)=-\frac { \chi(M)\pi^3 x}{4 } \int_{\C} e^{-\frac\pi 2|\xi|^2-\pi x^2}|\xi|^2 \left||\xi|^2- x^2 \right| d\xi$$$$+\frac { \chi(M) \pi ^2x}{2 } \int_{\C} e^{-\frac\pi 2|\xi|^2-\pi x^2} \left||\xi|^2- x^2 \right| d\xi  $$  
which completes the proof.

%We then have,  $$\Lambda_n^{-1}= \frac{1+n^{-1}\int_M %a_1dV_\omega}%{1+n^{-1}a_1}\begin{pmatrix}
 %\frac 12n^{-2}   & 0
%\\ 0
%&1
%\end{pmatrix}+\mbox{lower order terms}$$

\end{proof}

\end{document}